\documentclass[a4paper,10pt]{article}

\usepackage{graphicx}
\usepackage{geometry}
\usepackage{amsthm}
\usepackage{amssymb}
\usepackage{amsmath}
\usepackage{comment}
\usepackage{hyperref}
\usepackage{xcolor}
\usepackage{enumerate}
\usepackage{enumitem}
\usepackage{listings}
\usepackage{complexity}
\usepackage{blkarray}
\usepackage{lmodern}
\usepackage[english]{babel}
\usepackage{accents}
\usepackage[font=small,labelfont=bf]{caption}
\usepackage[font=small,labelfont=normalfont,labelformat=simple]{subcaption}
\graphicspath{{figs/}}

\newtheorem{theorem}{Theorem}
\newtheorem{definition}{Definition}[section]
\newtheorem{lemma}{Lemma}[section]
\newtheorem{corollary}[theorem]{Corollary}

\newtheorem{observation}{Observation}[section]
\newtheorem{proposition}[theorem]{Proposition}
\newtheorem{conjecture}[theorem]{Conjecture}
\newtheorem{remark}{Remark}[section]

\newcommand{\revvec}[1]{\accentset{\leftarrow}{#1}}
\newcommand{\dtw}{{\rm dtw}}

\newcommand{\bivec}[1]{\accentset{\leftrightarrow}{#1}}

\author
{
Raphael Steiner \thanks{Institute of Mathematics, Technische Universit\"at Berlin, Germany, email: \texttt{steiner@math.tu-berlin.de}.
Funded by DFG-GRK 2434 Facets of Complexity.}
}

\date{\today}

\title{Disjoint cycles with length constraints in digraphs of large connectivity or large minimum degree}

\begin{document}
\maketitle

\begin{abstract}
A conjecture by Lichiardopol~\cite{lich} states that for every $k \ge 1$ there exists an integer $g(k)$ such that every digraph of minimum out-degree at least $g(k)$ contains $k$ vertex-disjoint directed cycles of pairwise distinct lengths.

Motivated by Lichiardopol's conjecture, we study the existence of vertex-disjoint directed cycles satisfying length constraints in digraphs of large connectivity or large minimum degree.

Our main result is that for every $k \in \mathbb{N}$, there exists $s(k) \in \mathbb{N}$ such that every strongly $s(k)$-connected digraph contains $k$ vertex-disjoint directed cycles of pairwise distinct lengths.

In contrast, for every $k \in \mathbb{N}$ we construct a strongly $k$-connected digraph containing no two vertex- or arc-disjoint directed cycles of the same length. 

It is an open problem whether $g(3)$ exists. Here we prove the existence of an integer $K$ such that every digraph of minimum out- and in-degree at least $K$ contains $3$ vertex-disjoint directed cycles of pairwise distinct lengths. 
\end{abstract}

\section{Introduction}
Cycles are amongst the most fundamental graph objects, and a large body of work has dealt with understanding which conditions force a graph to contain cycles of certain types and lengths. In most cases, the conditions require the graphs to be sufficiently dense, which usually means that they should have either large minimum degrees or should be highly connected. 

The problem of packing vertex-disjoint cycles in undirected graphs of large minimum degree has received a lot of attention, see~\cite{chiba} for a survey on this topic. For example, for every $k \in \mathbb{N}$ an undirected graph whose minimum degree is sufficiently large in terms of $k$ contains
\begin{itemize}
\item $k$ vertex-disjoint cycles of equal lengths~\cite{alon, egawa, haggkvist}
\item $k$ vertex-disjoint cycles of pairwise distinct lengths~\cite{bensmail}
\item $k$ vertex-disjoint cycles of even lengths~\cite{chiba2}
\end{itemize}

In each case, the precise degree bounds required to obtain $k$ disjoint cycles with these properties are known.

In contrast, the analogous problems for directed graphs in large amounts remain open. In 1981, Bermond and Thomassen~\cite{bermondthomassen} stated the following conjecture.

\begin{conjecture}\label{con:bermondthomassen}
For every $k \in \mathbb{N}$, every digraph $D$ with $\delta^+(D) \ge 2k-1$ contains $k$ vertex-disjoint directed cycles.
\end{conjecture}

Thomassen~\cite{thom83} was the first to prove that for every $k$ there exists a finite number $f(k)$ such that all digraphs of minimum out-degree $f(k)$ contain $k$ disjoint directed cycles, and he showed that $f(k) \le (k+1)!$. This estimate was improved by Alon~\cite{alon} to $f(k) \le 64k$, and then further by Buci\'{c}~\cite{bucic} to $f(k) \le 18k$, which remains the state of the art.

Thomassen~\cite{thom83} also conjectured that his result may be strengthened in the sense that digraphs of sufficiently large minimum out-degree as a function of $k$ always contain $k$ disjoint directed cycles of equal lengths.
This conjecture was disproved by Alon~\cite{alon},
who gave a construction of digraphs of arbitrarily large minimum out-degree containing no two arc-disjoint directed cycles of equal lengths. 
On the other hand, Henning and Yeo~\cite{henningyeo} initiated the study of the question whether every digraph of sufficiently large minimum out-degree contains two disjoint directed cycles of \emph{distinct} lengths. This was resolved by Lichiardopol~\cite{lich}, who proved that every digraph $D$ with $\delta^+(D) \ge 4$ contains two vertex-disjoint directed cycles of distinct lengths. Lichiardopol conjectured the following far-reaching qualitative generalization of his result:

\begin{conjecture}[cf.~\cite{lich}]
For every $k \in \mathbb{N}$ there exists an integer $g(k) \in \mathbb{N}$ such that every digraph $D$ with $\delta^+(D) \ge g(k)$ contains $k$ vertex-disjoint directed cycles of pairwise distinct lengths.
\end{conjecture}

Remarkably, Lichiardopol's conjecture remains open even for $k=3$. In addition to Lichiardopol's proof of the existence of $g(2)$, some work has been done on forcing \emph{two} vertex-disjoint directed cycles of distinct lengths in more special classes of digraphs, see~\cite{gao, tan2014, tan2015, tan2017, tan2020} for some results on this topic. Bensmail, Harutyunyan, Le, Li and Lichiardopol~\cite{bensmail} took up Lichiardopol's conjecture and proved it in some special cases, namely for tournaments (orientations of complete graphs), for regular digraphs (all vertices have out- and in-degree exactly $r$ for some number $r \in \mathbb{N}$) and for digraphs of smaller order. In the latter two cases, their proof is probabilstic and constructs a partition of the digraph into $k$ vertex-disjoint subdigraphs with still large minimum out-degree. Whether or not such a partition exists for general digraphs is a fundamental open problem posed by several researchers, see for example~\cite{alon2}.
It therefore seems difficult to extend the proof methods of Bensmail et al.~to general digraphs, concretely since their proofs depend on (1) the inherent density property of tournaments or (2) the assumption that the digraph is \emph{balanced}, i.e.~that its maximum in-degree is upper-bounded by a function of its minimum out-degree, a requirement for the application of the Lov\'{a}sz Local Lemma in their probabilistic arguments. 

\paragraph{Our results.}
A natural strengthening of the condition of having large minimum degree is to require large connectivity. The main result of this paper verifies Lichiardopol's conjecture for digraphs of large (strong) connectivity. This is the first result guaranteeing arbitrarily many vertex-disjoint directed cycles of distinct lengths in digraphs which is applicable to general digraphs and does not require upper bounds on the maximum degrees.
\begin{theorem}\label{thm:mainconn}
For every $k \ge 1$ there exists an integer $s(k)$ such that every strongly $s(k)$-connected digraph contains $k$ vertex-disjoint directed cycles of pairwise distinct lengths. 
\end{theorem}
The relationship between minimum degree, strong connectivity and the existence of directed cycles of particular lengths in digraphs has received attention previously. For instance, a long-standing open problem posed by Lov\'{a}sz~\cite{lovasz} asked whether (1) every digraph of sufficiently large minimum out-degree contains an even directed cycle, and (2) whether every digraph of sufficiently large strong connectivity contains an even directed cycle. Both questions were resolved by Thomassen in~\cite{thom83} and~\cite{thom92}, answering (1) in the negative and (2) in the positive.

With Theorem~\ref{thm:mainconn} at hand, it is natural to ask whether a sufficiently high degree of strong connectivity also guarantees the existence of many vertex-disjoint directed cycles of equal length. 
We can show that the answer to this question is, maybe surprisingly, negative\footnote{Note that Proposition~\ref{prop:notwoofsamelength} is not implied by Alon's construction of digraphs having large minimum out-degree without two disjoint cycles of equal length, since the digraphs constructed in~\cite{alon} contain vertices of in-degree $1$, and hence are not strongly $2$-connected.}.
\begin{proposition}\label{prop:notwoofsamelength}
For every $k \in \mathbb{N}$ there exists a strongly $k$-connected digraph $D_k$ which contains no two arc-disjoint directed cycles (and hence no two vertex-disjoint cycles) of equal length.
\end{proposition}

Our next result makes progress on the question concerning the existence of $g(3)$, guaranteeing three vertex-disjoint directed cycles of distinct lengths in digraphs of large out- \emph{and} in-degree.

\begin{theorem}\label{thm:mainsem}
There exists an integer $K$ such that every digraph $D$ with $\delta^+(D), \delta^-(D) \ge K$ contains three vertex-disjoint directed cycles of pairwise distinct lengths.
\end{theorem}

Our proofs for Theorem~\ref{thm:mainconn} and Theorem~\ref{thm:mainsem} make use of the \emph{Directed Flat Wall Theorem}, a tool from structural digraph theory established recently by Giannopoulou, Kawarabayashi, Kreutzer and Kwon~\cite{kreutzer}. We believe it might be fruitful to investigate applications of this structural result to other packing problems in digraphs of similar nature. 

Theorem~\ref{thm:mainconn} shows that a counterexample to Lichiardopol's conjecture, if it exists, cannot contain any large well-connected parts. Intuitively, this means that the digraph does not have a 'rich' directed cycle structure, and other methods might apply to fully resolve the conjecture in this case. To illustrate this intuition, we prove Lichiardopol's conjecture for digraphs whose \emph{directed tree-width} is bounded. Directed tree-width is a parameter introduced by Johnson, Robertson, Seymour and Thomas~\cite{seymourdirectedwidth} as a generalization of undirected tree-width to directed graphs, which is supposed to be small for digraphs which have a 'sparse' directed cycle structure (we refer to Section~\ref{sec:prelim} for details on this parameter). 
\begin{proposition}\label{prop:boundedwidth}
Let $k, d \in \mathbb{N}$. Every digraph $D$ with $\delta^+(D)  > (d+2)(k-1)$ and directed tree-width at most $d$ contains $k$ vertex-disjoint directed cycles of pairwise distinct lengths.
\end{proposition}
In contrast to undirected graphs, digraphs of large minimum out-degree need not have large directed tree-width. In Section~\ref{sec:prelim}, Remark~\ref{rem:largedegreewidth1} we will show examples of digraphs having arbitrarily large minimum out-degree but directed tree-width one. Hence, Proposition~\ref{prop:boundedwidth} verifies Lichiardopol's conjecture for an additional non-trivial class of digraphs.

\paragraph{Structure of the paper.} In Section~\ref{sec:prelim} we introduce necessary notation and give important definitions of notions used in the paper, such a directed tree-width, and make some preliminary observations required later in the paper. Most importantly, we explain the necessary background and the statement of the \emph{Directed Flat Wall Theorem} from~\cite{kreutzer}. In Section~\ref{sec:proofs} we give the proofs of Theorem~\ref{thm:mainconn} and Theorem~\ref{thm:mainsem}. The directed flat wall theorem yields a natural division of our proofs into three different cases, which are prepared separately in the Subsections~\ref{sec:treewidth},~\ref{sec:kt} and~\ref{sec:flatwall}. In Subsection~\ref{sec:treewidth} we study digraphs of bounded directed tree-width and give the proof of Proposition~\ref{prop:boundedwidth}. In Subsection~\ref{sec:kt} we show that digraphs containing a large complete butterfly-minor contain many disjoint directed cycles of distinct lengths. In Subsection~\ref{sec:flatwall} we study digraphs containing large flat walls and prove sufficient conditions for the existence of disjoint directed cycles of different lengths. Finally, in Subsection~\ref{sec:thm12} we put the insights from the previous subsections together to conclude the proofs of Theorems~\ref{thm:mainconn} and~\ref{thm:mainsem}. In Section~\ref{sec:equicardinal}, we give the proof of Proposition~\ref{prop:notwoofsamelength}. We conclude with a conjecture in Section~\ref{sec:conc}.
\section{Preliminaries}\label{sec:prelim}
\paragraph*{Notation.} Digraphs in this paper are considered loopless, have no parallel arcs, but are allowed to have anti-parallel pairs of arcs (\emph{digons}). An arc (also called directed edge) with tail $u$ and head $v$ is denoted by $(u,v)$. The vertex-set (resp. arc-set) of a digraph $D$ is denoted by $V(D)$ (resp. $A(D)$). For $X \subseteq V(D)$, we denote by $D[X]$ the  subdigraph of $D$ induced by $X$. For a set $X$ of vertices or arcs in $D$, we denote by $D-X$ the subdigraph obtained by deleting the objects in $X$ from $D$.
We use $\bivec{K}_k$ to denote the complete digraph of order $k$, i.e., in which every ordered pair of distinct vertices appears as an arc.
For a digraph $D$ and a vertex $v \in V(D)$, we let $N_D^+(v),$ $N_D^-(v)$ denote the out- and in-neighborhood of $v$ in $D$ and $d_D^+(v)$, $d_D^-(v)$ their respective sizes (we drop the subscript $D$ if it is clear from context). We denote by $\delta^+(D)$, $\delta^-(D)$, $\Delta^+(D)$, $\Delta^-(D)$ the minimum or maximum  out- or in-degree of $D$, respectively. 
A \emph{directed walk} in a digraph is an alternating sequence $v_1,e_1,v_2,\ldots,v_{k-1},e_{k-1},v_k$ of vertices and arcs such that $e_i=(v_i,v_{i+1})$ for all $1 \leq i \leq k-1$. The walk is called {\em closed} if $v_k = v_1$. 
 
If $v_1,\ldots,v_k$ are pairwise distinct, then we say that $P$ is a \emph{directed path} or \emph{dipath} from $v_1$ to $v_k$ ({\em $v_1$-$v_k$-dipath} for short). We will call $v_1$ the first vertex of $P$, $v_2$ the second vertex of $P$, $v_k$ the last vertex of $P$, etc., and refer to $v_1, v_k$ as the \emph{endpoints} of $P$. By $V(P)=\{v_1,\ldots,v_k\}$ we denote the vertex-set of $P$.
Given two distinct vertices $x \neq y$ on a directed path $P$, we denote by $P[x,y]=P[y,x]$ the subpath of $P$ with endpoints $x$ and $y$. A vertex $v$ in a digraph $D$ is said to be \emph{reachable} from a vertex $u$ if there exists a $u$-$v$-dipath. For a pair of dipaths $P,Q$ such that the first vertex of $Q$ is the last vertex of $P$, we denote by $P \circ Q$ the concatenation of $P$ and $Q$, i.e. the directed walk obtained by first traversing $P$ and then traversing $Q$. When $P$ (resp. $Q$) consists of a single arc $(x,y)$, we will sometimes write $(x,y) \circ Q$ (resp. $P \circ (x,y)$) instead of $P \circ Q$. 

A \emph{directed cycle} is a cycle with all arcs oriented consistently in one direction. For a directed cycle $C$ and two distinct vertices $x,y \in V(C)$, we denote by $C[x,y]$ the segment of $C$ which forms a dipath from $x$ to $y$. By $|C|$ we denote the length of $C$.

A digraph $D$ is called \emph{strongly connected} if for every ordered pair $(x,y) \in V(D)\times V(D)$ of vertices, $x$ can reach $y$ in $D$. The maximal strongly connected subgraphs of a digraph $D$ are called \emph{(strong) components} and induce a partition of $V(D)$. For a natural number $k \ge 2$, a digraph $D$ is called \emph{strongly $k$-vertex-connected} (or simply strongly $k$-connected) if $|V(D)| \ge k+1$ and for every set $S$ of at most $k-1$ vertices (arcs) of $D$, the digraph $D-K$ is strongly connected. An \emph{out-arborescence} is a directed rooted tree in which all arcs are directed away from its root.

Next we state the version of the directed flat wall theorem we need from~\cite{kreutzer}, and in order to do so we recall several definitions from structural digraph theory, which will reappear in our proofs. We start with the concept of directed tree-width. This is a generalization of the highly influential notion of tree-width to directed graphs, which was introduced by Johnson, Robertson, Seymour and Thomas~\cite{seymourdirectedwidth}. The basic idea of directed tree-width is to measure the richness of the directed cycle structure of a digraph. For instance, digraphs of directed tree-width $0$ are exactly the acyclic digraphs. We now give the formal definitions.

For two distinct vertices $r, r' \in V(T)$ in an out-arborescence $T$ we write $r<r'$ if $r'$ is reachable from $r$ in $T$. We write $r \le r'$ to express that $r=r'$ or $r<r'$. If $e \in A(T)$ whose head is $r$, then we write $e<r'$ if $r \le r'$.
For a digraph $D$ and a subset $Z \subseteq V(D)$, we call $S \subseteq V(D)\setminus Z$ a \emph{$Z$-normal} set if the vertex set of every directed walk in $D-Z$ which starts and ends in $S$ is completely contained in $S$. Note that every $Z$-normal set is a union of certain strong components of $D-Z$.
\begin{definition}[cf.~\cite{seymourdirectedwidth}]
A \emph{directed tree decomposition} of a digraph $D$ is a triple $(T,\beta,\gamma)$, where $T$ is an out-arborescence, $\beta:V(T) \rightarrow 2^{V(D)}$, $\gamma: A(T) \rightarrow 2^{V(D)}$ are functions such that
\begin{enumerate}[label=(\roman*)]
\item $\{\beta(t)|t \in V(T)\}$ is a partition of $V(D)$ into non-empty sets and
\item if $e \in A(T)$, then $\bigcup\{\beta(t)|t \in V(T), e<t\}$ is $\gamma(e)$-normal.
\end{enumerate}
For every $t \in V(T)$ we denote $$\Gamma(t):=\beta(t) \cup \bigcup \{\gamma(e)| e \in A(T), e \sim t\},$$ where $e \sim t$ expresses that $e$ is incident with $t$ in $T$. 
Then the \emph{width} of the tree-decomposition $(T,\beta,\gamma)$ is defined as $w=\max_{t \in V(T)}{|\Gamma(t)|}-1.$ The \emph{directed tree-width} of $D$, denoted by $\dtw(D)$, is the smallest possible width of a directed tree-decomposition of $D$.
\end{definition}
The sets $\beta(t)$ are called \emph{bags}, while the sets $\gamma(e)$ are the \emph{guards} of the tree-decomposition.

Let us give at this point the examples demonstrating that, in contrast to undirected graphs, digraphs of large minimum out-degree need not have large directed tree-width.
\begin{remark}\label{rem:largedegreewidth1}
For every $k \in \mathbb{N}$ there exists a digraph $F_k$ such that $\delta^+(F_k)=k$ and $\dtw(F_k)=1$.
\end{remark}
\begin{proof}
Let $k \in \mathbb{N}$ be fixed. Let us denote by $T_k$ the unique $k$-ary out-arborescence of depth $k$ (that is, every non-leaf vertex has $k$ children, and every leaf has distance $k$ from the root). Let $r$ denote the root of $T_k$. Let $F_k$ be the digraph obtained from $T_k$ by adding the arc $(u,v)$ for every leaf $u \in V(T_k)$ and every vertex $v \in V(T_k)$ such that $u>v$ in $T_k$. Since every leaf of $T_k$ has $k$ ancestors, $D_k$ has minimum out-degree $k$. To see that $\dtw(F_k)=1$, let us consider the directed tree-decomposition $(T,\beta,\gamma)$ of $F_k$, where $T:=T_k$, $\beta(t):=\{t\}$ for every $t \in V(T)$ and $\gamma(e):=\{\text{tail}(e)\}$ for every arc $e \in A(T)$. To see that this forms a directed tree-decomposition, let $e=(u,v) \in A(T)$ be arbitrary. Then $S:=\bigcup\{\beta(t)|t \in V(T), e<t\}$ is the set of vertices in $F_k$ contained in the subtree of $T_k$ rooted at $v$. We need to show that $S$ is $\gamma(e)=\{u\}$-normal in $D$. However, in $F_k-u$, there exists no arc starting in a vertex outside $S$ and ending in $S$, which directly shows that a directed walk in $F_k-u$ ending in $S$ must be contained in $S$. This shows that $(T,\beta,\gamma)$ indeed is a directed tree-decomposition. By definition of $\beta$ and $\gamma$ we have $\Gamma(t)=\beta(t) \cup \bigcup\{\gamma(e)|e \in A(T), e \sim t\}=\{t,\text{parent}(t)\}$ for every vertex $t \in V(T)\setminus \{r\}$, and $\Gamma(r)=\{r\}$. It follows that the width of $(T,\beta,\gamma)$ is $2-1=1$, proving that $\dtw(F_k)=1$.
\end{proof}

We will need a relation between directed tree-width and so-called \emph{havens} proved in~\cite{seymourdirectedwidth}.
\begin{definition}
Let $D$ be a digraph and let $k \in \mathbb{N}$. A \emph{haven} of order $k$ for $D$ is a function $h$, assigning to every set $X \subseteq V(D)$ of size $|X|<k$ a set $h(X) \subseteq V(D)\setminus X$, which forms a strong component of $D-X$, and which has the property that $h(X) \supseteq h(Y)$ for every $X \subseteq Y \subseteq V(D)$ with $|Y|<k$.
\end{definition}
\begin{theorem}[cf.~\cite{seymourdirectedwidth}, Theorem 3.1]\label{thm:largehavenforceslargedtw}
Let $D$ be a digraph. If $D$ admits a haven of order $k$, then $\dtw(D) \ge k-1$.
\end{theorem}
Let us observe the following consequence of Theorem~\ref{thm:largehavenforceslargedtw} for later use.
\begin{corollary}\label{cor:highconnhighwidth}
Let $k \in \mathbb{N}$, and let $D$ be a strongly $k$-connected digraph. Then $\dtw(D) \ge k$.
\end{corollary}
\begin{proof}
Let us define $h:\binom{V(D)}{\le k} \rightarrow 2^{V(D)}$ as follows: or every $X \subseteq V(D)$ with $|X| \le k$, we let $h(X):=V(D) \setminus X$, if $|X|<k$, while $h(X)$ is defined as the vertex-set of an arbitrarily chosen strong component of $D-X$ if $|X|=k$. We claim that $h$ is a haven of order $k+1$ for $D$. Indeed, for every $X \subseteq V(D)$ such that $|X|<k$, $D-X$ is strongly connected and hence $h(X)=V(D)\setminus X$ is the unique strong component of $D-X$. If $|X|=k$, then by definition $h(X)$ is a strong component of $D-X$. Let now $X \subseteq Y \subseteq V(D)$ be arbitary such that $|Y| \le k$. If $|Y|<k$, then we clearly have $h(X)=V(D)\setminus X \supseteq V(D)\setminus Y$. If $|Y|=k$, then either $X=Y$ and $h(X)=h(Y)$ or $|X|<k$ and $h(X)=V(D)\setminus X \supseteq h(Y)$, since $h(Y)$ is a strong component of $D-Y \subseteq D-X$. This shows that indeed, $h$ is a haven of order $k+1$ in $D$ and the claim follows by applying Theorem~\ref{thm:largehavenforceslargedtw}.
\end{proof}

As for undirected tree-width, there exists a duality theorem for directed tree-width and the containment of large grid-like substructures, known as the \emph{directed grid theorem}. This was proved in a breakthrough-result by Kawarabayashi and Kreutzer~\cite{kawkreut}. The \emph{directed flat wall theorem} can be seen as a further strengthening of this result. In order to make both theorems precise, we need to specify what we mean by 'containment' and by 'grid-like substructure'.

\begin{definition}[Topological minor]
Given a digraph $D$, a \emph{subdivision} of $D$ is any digraph $D'$ obtained from $D$ by replacing every arc $e=(u,v) \in A(D)$ by a directed path $P_{e}$ starting in $u$ and ending in $v$, such that the paths $P_e, e \in A(D)$ are pairwise internally vertex-disjoint and intersect $V(D)$ only in their endpoints. Given a subdivision $D'$, the vertices in $D'$ originally contained in $D$ are called \emph{branch vertices}, while the vertices internal to a path $P_e, e \in A(D)$ are called \emph{subdivision vertices}. Let $D_1, D_2$ be digraphs. We say that $D_2$ is a \emph{topological minor} of $D_1$, if $D_1$ contains a subdivision of a digraph isomorphic to $D_2$ as a subdigraph.
\end{definition}
\begin{definition}[Butterfly minor]
Let $D$ be a digraph. An arc $e=(u,v) \in A(D)$ is called \emph{contractible} if $d_D^+(u)=1$ or $d_D^-(v)=1$, i.e. if $e$ is the only arc leaving $u$ or the only arc entering $v$. If $e$ is contractible, we denote by $D/e$ the digraph obtained from $D$ by contracting $e$, i.e., identifying $u$ and $v$ and joining their neighborhoods, ignoring loops and multiple arcs in the same direction created by this process. More precisely, we define $D/e$ by
$$V(D/e):=V(D)\setminus \{u\}, A(D/e):=A(D-u) \cup \{(x,v)|x \in N_D^-(u)\setminus \{v\}\},$$
if $d_D^+(u)=1$, and
$$V(D/e):=V(D)\setminus \{v\}, A(D/e):=A(D-v) \cup \{(u,x)|x \in N_D^+(v)\setminus \{u\}\},$$
if $d_D^+(u) \ge 2$ and $d_D^-(v)=1$. 

A digraph $D'$ is called a \emph{butterfly minor} of another digraph $D$ if $D'$ is isomorphic to a digraph which can be obtained from $D$ via a finite sequence of arc-deletions, vertex-deletions, and contractions of contractible arcs (in arbitrary order).
\end{definition}

We remark that it can be proved by induction that if $F$ is a digraph of maximum degree $3$ such that $\Delta^+(F), \Delta^-(F) \le 2$, then $F$ is a butterfly-minor of $D$ iff it is a topological minor of $D$.

\begin{definition}[Cylindrical Grid, cf.~\cite{kreutzer}]
A \emph{cylindrical grid of order $k$} for $k \in \mathbb{N}, k \ge 1$ is a digraph $G_k$ consisting of $k$ directed cycles $C_1,\ldots,C_k$, pairwise vertex-disjoint, together with a set of $2k$ pairwise vertex-disjoint directed paths $P_1,\ldots,P_{2k}$, such that 
\begin{itemize}
\item each path $P_i$ has exactly one vertex in common with each cycle $C_j$ and both endpoints of $P_i$ are in $V(C_1) \cup V(C_k)$,
\item the paths $P_1,\ldots,P_{2k}$ appear on $C_i$ in this order, and
\item for odd $i$ the cycles $C_1,\ldots,C_k$ occur on all $P_i$ in this order and for even $i$ they appear in reverse order $C_k,\ldots,C_1$.
\end{itemize}
\end{definition}
\begin{definition}[Cylindrical Wall, cf.~\cite{kreutzer}]
An \emph{elementary cylindrical wall of order $k$} for $k \in \mathbb{N}$ is the planar digraph $W_k$ obtained from the cylindrical grid $G_k$ of order $k$ by replacing every vertex $v \in V(G_k)$ by two new vertices $v_s, v_t$ connected by an arc $(v_s,v_t)$ such that for every arc $(v,w) \in A(G_k)$, we have the corresponding arc $(v_t,w_s) \in A(W_k)$. A \emph{cylindrical wall of order $k$} is any digraph isomorphic to a subdivision of $W_k$.
\end{definition}
We remark that $G_k$ is a butterfly-minor of $W_k$ obtained by contracting all the split-edges $(v_s,v_t), v \in V(G_k)$, while it can be observed that $W_k \subseteq G_{2k}$ for every $k \in \mathbb{N}$. It is convenient to imagine grids, elementary cylindrical walls and their subdivisions as embedded in the plane, as depicted in Figure~\ref{fig:walls}. Every cylindrical wall $W$ of order $k$ contains in this canonical embedding $k$ pairwise vertex-disjoint 'vertical' directed cycles $Q_1,\ldots,Q_k$, as well as $2k$ pairwise vertex-disjoint directed horizontal paths $P_i^1,P_i^2, i=1,\ldots,k$ which are alternately directed from left to right or from right to left. To make references to specific vertices in a cylindrical wall easier, we assign coordinates to the branch vertices of a cylindrical wall of order $k$ based on its canonical embedding as follows: For every $1 \le i \le k$, and $1 \le j \le 2k$, the $j$-th branch vertex on the directed path $P_i^1$ receives coordinates $(j,2i-1)$, while the $j$-th branch vertex on the directed path $P_i^2$ receives coordinates $(2k+1-j,2i)$.
\begin{figure}[h]
\centering
\includegraphics[scale=0.4]{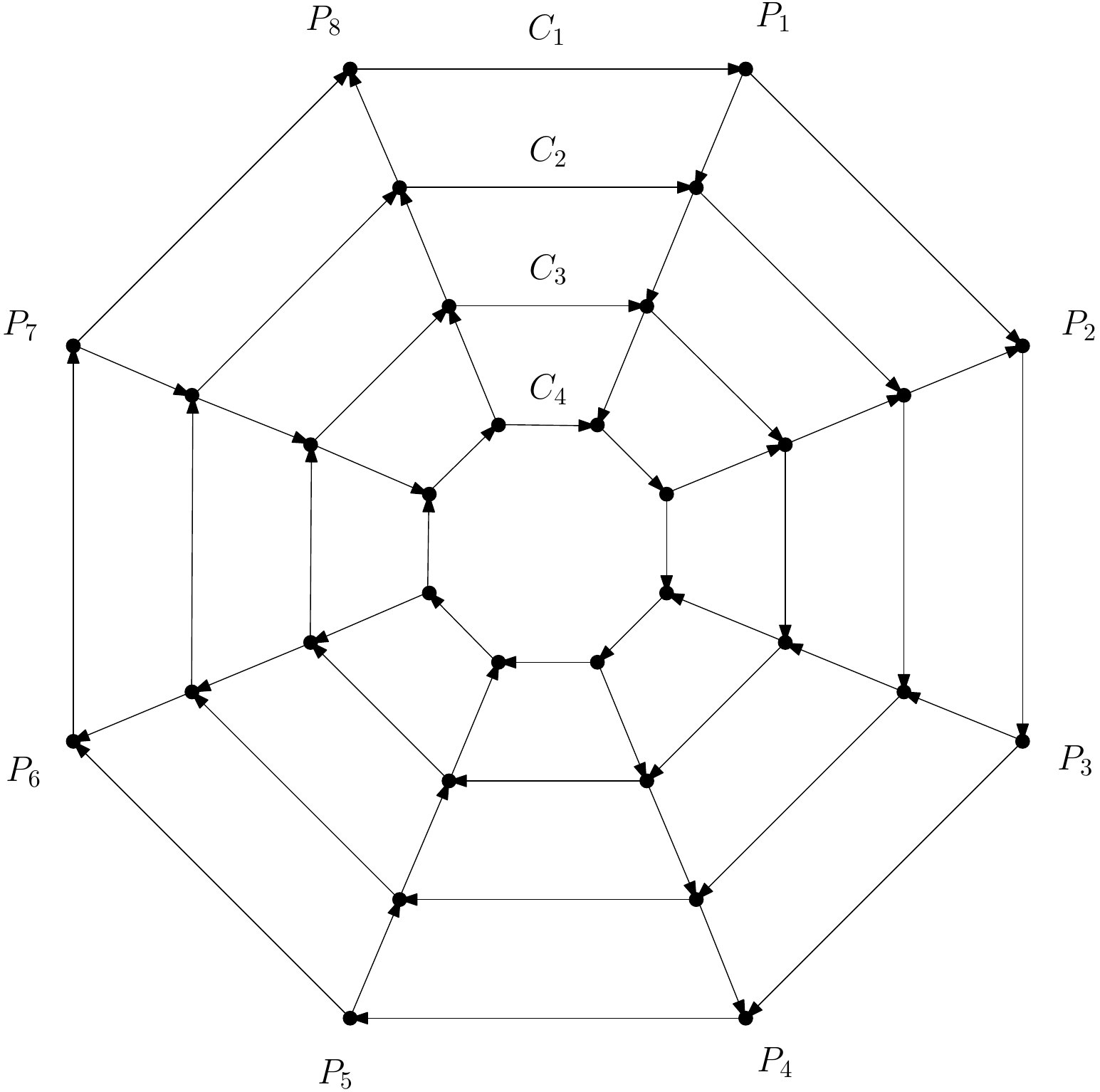} \includegraphics[scale=0.55]{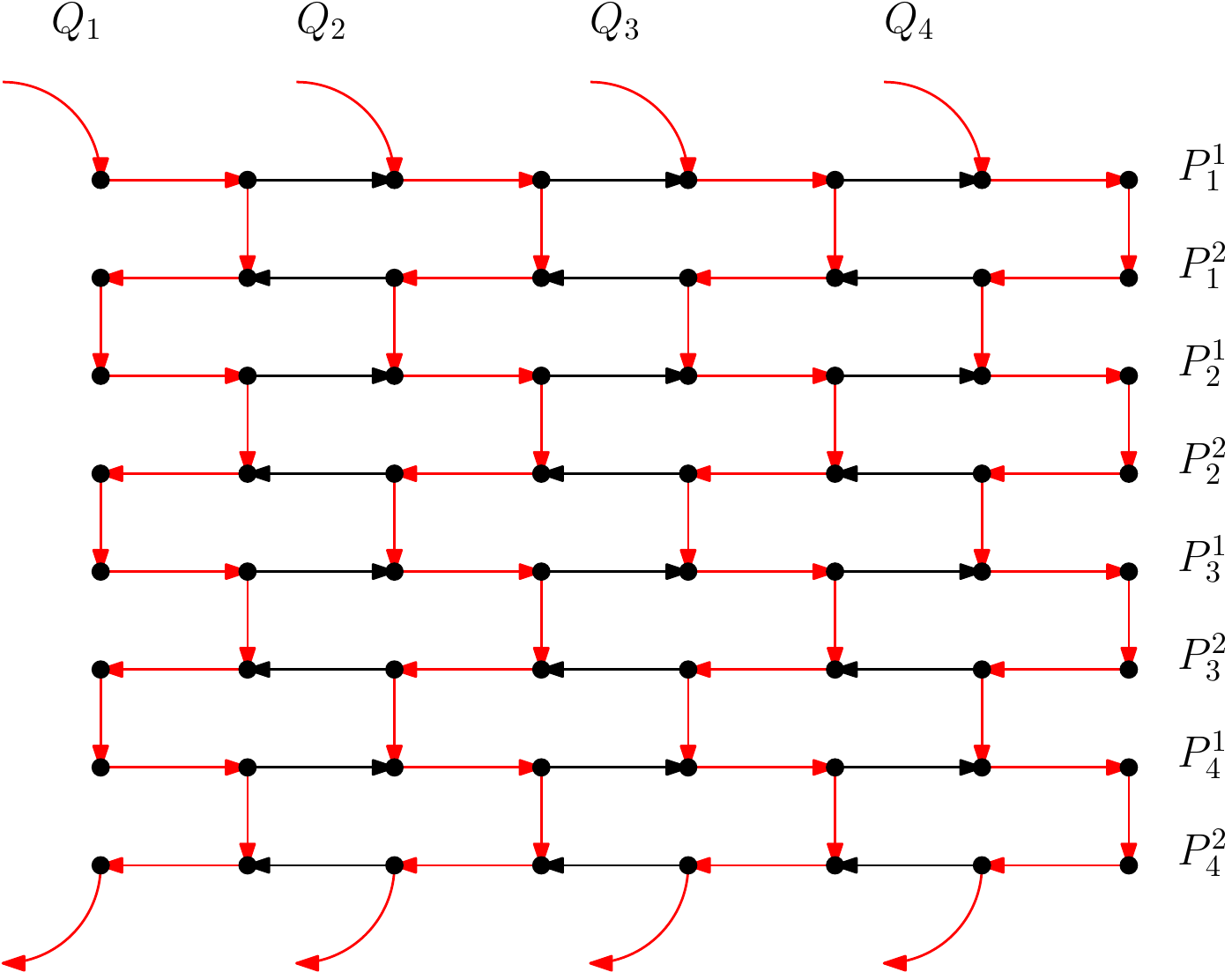}
\caption{Left: The cylindrical grid $G_4$ of order $4$. Right: The elementary cylindrical wall $W_4$ in its canonical depiction with the cycles $Q_1,\ldots,Q_4$ (marked red) and the horizontal paths $P_1^1,P_1^2\ldots,P_4^1,P_4^2$. The red half-edges at the top and the bottom of the wall indicate arcs connecting the paths $P_1^1$ and $P_4^2$.}\label{fig:walls}
\end{figure}

\begin{definition}[Perimeter, Bricks cf.~\cite{kreutzer}]
Let $W$ be a cylindrical wall of order $k$. The \emph{perimeter} $\textbf{per}(W)$ of $W$ is defined as the union $V(Q_1) \cup V(Q_k)$, where $Q_1$ and $Q_k$ are the first resp. last vertical directed cycles in the canonical embedding of $W$. We further define the \emph{interior} of $W$ by $\textbf{int}(W):=V(W) \setminus \textbf{per}(W)$.

A \emph{brick} of $W$ is a cycle in the canonical embedding of $W$ induced by the boundary of an inner face of $W$ (i.e., distinct from the two faces bounded by the cycles $Q_1$ and $Q_k$). Every brick contains exactly $6$ branch vertices of $W$.
\end{definition}

We are now ready to state the Directed Grid Theorem.

\begin{theorem}[cf.~\cite{kawkreut}]\label{thm:gridthm}
For every $k \in \mathbb{N}$ there exists an integer $d(k)$ such that every digraph $D$ with $\dtw(D) \ge d(k)$ contains the cylindrical grid of order $k$ as a butterfly-minor.
\end{theorem}
Since $W_k \subseteq G_{2k}$, every digraph containing a $G_{2k}$-butterfly-minor also contains a $W_k$-minor and hence a subdivision of $W_k$, since $W_k$ is subcubic. We conclude that every digraph sufficiently high directed tree-width (at least $d(2k)$) contains a wall of order $k$ as a subdigraph. 

To state the strengthening of this statement guaranteed by the Directed Flat Wall Theorem, we need a definition of \emph{flatness} of a wall contained in a digraph. Giannopoulou et al.~(cf.~\cite{kreutzer}, Definition 2.14) give a complex definition including $3$ different items. For our purposes, only the properties of a flat wall guaranteed by the second item of their definition are required. To not complicate matters unnecessarily, we use the following weakened definition of flatness. Intuitively, the definition asserts that directed paths in the digraph do not jump far across the wall.
\begin{definition}
Let $D$ be a digraph, and let $W \subseteq D$ be a cylindrical wall. We say that $W$ is \emph{weakly flat in $D$}, if for every directed path in $D$ which is internally vertex-disjoint from $\textbf{int}(W)$ with both endpoints $x, y$ contained in $\textbf{int}(W)$, there exists a brick $B$ of $W$ such that $x,y \in V(B)$.
\end{definition}

We can now state a weak version of the Directed Flat Wall Theorem from~\cite{kreutzer} as follows.

\begin{theorem}[cf.~\cite{kreutzer}, Theorem 2.3]\label{thm:flatwallstrong}
For every $k, t \in \mathbb{N}$ there are integers $d(k,t)$ and $a(t)$ such that for every digraph $D$ at least one of the following is true:
\begin{enumerate}[label=(\roman*)]
\item $\dtw(D)<d(k,t)$,
\item $D$ contains $\bivec{K}_t$ as a butterfly-minor,
\item there is a set $X \subseteq V(D)$ of order $|X| \le a(t)$ and a cylindrical wall $W \subseteq D-X$ of order $k$ which is weakly flat in $D-X$.
\end{enumerate}
\end{theorem}
\section{Finding Disjoint Dicycles of Distinct Lengths}\label{sec:proofs}
In this section we give the proofs of Theorem~\ref{thm:mainconn} and Theorem~\ref{thm:mainsem}.

The structure of the section follows the three possible cases given by Theorem~\ref{thm:flatwallstrong}:
\begin{enumerate}[label=(\roman*)]
\item digraphs of bounded directed tree-width,\item digraphs containing a large complete butterfly-minor, \item digraphs containing a weakly flat wall. 
\end{enumerate}
In the last subsection we then prove Theorem~\ref{thm:mainconn} and Theorem~\ref{thm:mainsem} by applying Theorem~\ref{thm:flatwallstrong} and using the insights from the three previous subsections.

At this point it is worth pointing out why our proof strategy needs to make use of the Direced Flat Wall Theorem~\ref{thm:flatwallstrong} to make progress on Lichiardopol's conjecture, and why its weakening, the Directed Grid Theorem~\ref{thm:gridthm}, is not as helpful. The main reason for this is that while walls of large order contain many vertex-disjoint directed cycles, they might not even contain two directed cycles whose lengths are different, as shown by the Remark~\ref{rem:wallallsamelength} below. Hence, it seems unlikely that Theorem~\ref{thm:gridthm} will be helpful in obtaining disjoint directed cycles of distinct lengths, since it gives no information concerning the relation between the wall it contains and the rest of the digraph. This disadvantage is improved by the Directed Flat Wall Theorem, which restricts the ways in which directed paths can intersect a weakly flat wall.

\begin{remark}\label{rem:wallallsamelength}
For every $k \in \mathbb{N}$ there exists a cylindrical wall of order $k$ containing no two directed cycles of different lengths.
\end{remark}
\begin{proof}
A subdivision of $W_k$ is determined by the lengths of the subdivision-paths replacing its arcs. Hence, we may as well give an assignment $w:A(W_k) \rightarrow \mathbb{N}$ of positive integers to the arcs of $W_k$ such that the sum of the labels on any directed cycle is the same. To do so, let $R$ denote the set of arcs in $W_k$ starting in $V(P_k^2)$ and ending in $V(P_1^1)$, and observe that every directed cycle in $W_k$ contains \emph{exactly} one arc in $R$. As a consequence, the digraph $W_k-R$ is acyclic. Let $v_1,\ldots,v_{4k^2}$ be a topological ordering of this digraph, i.e., such that $(v_i,v_j) \notin A(W_k-R)$ whenever $i>j$. Let us define $w(e):=j-i$ for every $e=(v_i,v_j) \in A(W_k)\setminus R$. Let $L:=4k^2$. By definition of the arc-weighting $w$, for every directed path $P=v_{i_1},v_{i_2},\ldots,v_{i_\ell}$ in $W_k-R$, it now holds that $\sum_{e \in A(P)}{w(e)}=i_\ell-i_{\ell-1}+i_{\ell-1}-i_{\ell-2} \pm \ldots +i_2-i_1=i_\ell-i_1<L$. In particular, the total weight of any directed path in $A(W_k)$ depends only on its endvertices and is smaller than $L$. We conclude that for every arc $e=(u,v) \in R$, there exists a number $L_{uv}\in \mathbb{N}, L_{uv}<L$ such that every directed path in $W_k-R$ starting in $v$ and ending in $u$ has total arc-weight $L_{uv}$. Let us now put $w(e):=L-L_{uv} \in \mathbb{N}$ for every $e=(u,v) \in R$. 

Every directed cycle $C$ in $W_k$ intersects $R$ in exactly one arc $e=(u,v) \in R$. Then $C-e$ is a directed $v$-$u$-path in $W_k-R$, and we conclude $\sum_{e \in A(C)}{w(e)}=L-L_{uv}+L_{uv}=L$. This shows that $w$ is an arc-weighting of $W_k$ with positive integers in which all directed cycles have total arc-weight equal to a number $L \in \mathbb{N}$. This concludes the proof.
\end{proof}
\subsection{Digraphs of Bounded Directed Tree-Width}\label{sec:treewidth}
We start this subsection with the crucial definition of $k$-trains. These are digraphs containing many distinct cycle lengths, which are useful for finding disjoint directed cycles of distinct lengths.
\begin{definition}
Let $k \in \mathbb{N}$. A \emph{$k$-train} is a digraph consisting of a directed path $P$ with vertex-trace $u_0,u_1,\ldots,u_\ell$ directed from $u_0$ to $u_\ell$ together with $k$ arcs of the form $(u_\ell,u_{\ell_j}), j=1,\ldots,k$, where $0=\ell_1<\ell_2<\cdots<\ell_k<\ell$. 
\end{definition}
Note that every $k$-train is strongly connected and contains $k$ directed cycles of pairwise distinct lengths, namely the cycles through the arcs $(u_\ell,u_{\ell_j}), j=1,\ldots,k$. This fact has the following useful consequence.
\begin{observation}\label{obs:disjointktrains}
If a digraph $D$ contains $k$ vertex-disjoint $k$-trains as subdigraphs, then $D$ contains $k$ vertex-disjoint directed cycles of pairwise distinct lengths.
\end{observation}
\begin{proof}
Let $D_1,\ldots,D_k \subseteq D$ be vertex-disjoint $k$-trains. For $i=1,2,3,\ldots,k$ we successively pick a directed cycle $C_i$ in $D_i$ whose length is distinct from the lengths of the already chosen directed cycles $C_1,\ldots,C_{i-1}$, which is possible since $D_i$ contains directed cycles of $k$ different lengths. Eventually the process returns $k$ vertex-disjoint directed cycles of pairwise distinct lengths.
\end{proof}
The next observation shows that $k$-trains exist in digraphs of large out-degree.
\begin{observation}\label{obs:existsktrain}
Let $D$ be a digraph with $\delta^+(D) \ge k$. Then $D$ has a $k$-train as a subdigraph.
\end{observation}
\begin{proof}
Let $P$ be the longest directed path in $D$. Then $P$ together with the arcs leaving its end-vertex $u$ contains a $k$-train, as all the out-neighbors of $u$ are contained in $V(P)$.
\end{proof}

Observation~\ref{obs:disjointktrains} motivates studying packings of disjoint $k$-trains in digraphs. In this spirit, Lemma~\ref{lemma:erdosposaboundedwidth} below shows that within digraphs of bounded directed tree-width, $k$-trains have the so-called \emph{Erd\H{o}s-P\'{o}sa} property, i.e., a qualitative relation between the maximum number of disjoint $k$-trains within such a digraph and the smallest size of a vertex-set hitting all $k$-trains. For the proof we need a result from~\cite{erdosposa}. In the following, given a directed tree-decomposition $(T,\beta,\gamma)$ of a digraph, we use the notation $$\beta(\ge t):=\bigcup_{t' \in V(T), t' \ge t}{\beta(t')}$$ for the set of the vertices of $D$ contained in the bags of the sub-arborescence of $T$ rooted at $t$.
\begin{lemma}[cf. Lemma 3.6 in~\cite{erdosposa}]\label{lemma:intersect}
Let $(T,\beta,\gamma)$ be a directed tree-decomposition of a digraph $D$ and let $H$ be a strongly connected subdigraph of $D$. Let $r$ be the root of $R$, and let $t^\ast \in V(T)$ be a node in $T$ of maximal distance from $r$ in $T$ such that $\beta(\ge t^\ast) \supseteq V(H)$. Then for every $t \in V(T)$ with $t \ge t^\ast$, we have $\Gamma(t) \cap V(H) \neq \emptyset$.
\end{lemma}
\begin{lemma}\label{lemma:erdosposaboundedwidth}
Let $k, d \in \mathbb{N}$. The following holds for every digraph $D$ with $\dtw(D) \le d$:

For every $\ell \in \mathbb{N}$, $D$ contains $\ell$ vertex-disjoint $k$-trains as subdigraphs, or there is a subset $X \subseteq V(D)$ with $|X| \le (d+1)(\ell-1)$ such that $D-X$ contains no $k$-train.
\end{lemma}
\begin{proof}
Let us proof the claim by induction on $\ell$. If $\ell=1$ the claim holds trivially. Now suppose that $\ell \ge 2$ and the claim has been proved for $\ell-1$. Let $(T,\beta,\gamma)$ be a directed tree-decomposition of $D$ of width at most $d$. We may assume w.l.o.g.~that $D$ contains at least one $k$-train as a subdigraph, for otherwise the claim holds trivially. Since the sets $\beta(t), t \in V(T)$ partition $V(D)$, there exists at least one $t \in V(T)$ such that $D[\beta(\ge t)]$ contains a $k$-train. Let $t_0 \in V(T)$ be a vertex with this property maximizing the distance to the root in $T$. Let $D':=D-\beta(\ge t_0)$. Since $\dtw(\cdot)$ is monotone under taking subgraphs, we have $\dtw(D') \le \dtw(D) \le d$. By the induction hypothesis, there exist $\ell-1$ disjoint $k$-trains in $D'$, or we can hit all $k$-trains contained in $D'$ by a set $X' \subseteq V(D')$ with $|X'| \le (d+1)(\ell-2)$. In the first case, we can simply join a collection of $\ell-1$ vertex-disjoint $k$-trains in $D'$ by an (arbitrarily chosen) $k$-train contained in $D[\beta(\ge t_0)]$ to obtain a collection of $\ell$ vertex-disjoint $k$-trains in $D$, which in this case concludes the proof. So we may assume that we are in the second case and that the set $X'$ with the stated property exists. We now define $X:=X' \cup \Gamma(t_0) \subseteq V(D)$, noting that $|X| \le |X'|+|\Gamma(t_0)| \le (d+1)(\ell-2)+(d+1)=(d+1)(\ell-1)$, since $(T,\beta,\gamma)$ has width at most $d$. We claim that $D-X$ contains no $k$-trains as subdigraphs, which if true, proves the inductive claim. Suppose towards a contradiction that there exists a $k$-train $H$ in $D$ such that $V(H) \cap X=\emptyset$. Since $D'-X'$ contains no $k$-trains, we must have $V(H) \cap \beta(\ge t_0) \neq \emptyset$. Let $t^\ast \in V(T)$ be such that $\beta(\ge t^\ast) \supseteq V(H)$ and such that the distance of the root to $t^\ast$ in $T$ is maximized. There are three possible cases concerning the relationship of $t_0$ and $t^\ast$ in $T$: (1) $t_0$ and $t^\ast$ are incomparable, (2) $t_0<t^\ast$ and (3) $t_0 \ge t^\ast$. In case (1), we have that $\beta(\ge t_0)$ and $\beta(\ge t^\ast)$ are disjoint, yielding a contradiction since $V(H) \cap \beta(\ge t_0) \neq \emptyset$ and $V(H) \subseteq \beta(\ge t^\ast)$. In case (2), $t^\ast$ is a vertex in $V(T)$ which is further from the root than $t_0$, but still $\beta(\ge t^\ast)$ contains a $k$-train (namely $H$). This yields a contradiction to the definition of $t_0$. In case (3), noting that $H$ is a strongly connected subdigraph of $D$, we can apply Lemma~\ref{lemma:intersect} to find that $\Gamma(t_0) \cap V(H) \neq \emptyset$. This contradicts the facts $V(H) \cap X=\emptyset$ and $\Gamma(t_0) \subseteq X$. Since all three cases lead to a contradiction, our assumption was wrong, $X$ indeed hits all the $k$-trains in $D$. This concludes the inductive proof of the claim.
\end{proof}
From the above Lemma, we can now conclude that Lichiardopol's conjecture holds for digraphs of bounded directed tree-width and give the proof of Proposition~\ref{prop:boundedwidth}.

\begin{proof}[Proof of Proposition~\ref{prop:boundedwidth}]
Let $D$ be any given digraph such that $\delta^+(D) > (d+2)(k-1)$ and $\dtw(D) \le d$. We claim that $D$ must contain $k$ vertex-disjoint $k$-trains. Towards a contradiction suppose not, then by Lemma~\ref{lemma:erdosposaboundedwidth} there exists $X \subseteq V(D)$ with $|X| \le (d+1)(k-1)$ such that $D-X$ contains no $k$-train as a subdigraph. However, we have $\delta^+(D-X) \ge \delta^+(D)-|X|$ $\ge (d+2)(k-1)+1-(d+1)(k-1)=k$. Therefore, by Observation~\ref{obs:existsktrain} $D-X$ must contain a $k$-train, a contradiction.
Hence, $D$ indeed contains $k$ vertex-disjoint $k$-trains. The claim now follows from Observation~\ref{obs:disjointktrains}.
\end{proof}
\subsection{Digraphs Containing a Large Complete Minor}\label{sec:kt}
In this subsection, we show that every digraph containing as a butterfly-minor a complete digraph of sufficiently large order must also contain $k$ disjoint directed cycles of distinct lengths. To achieve this goal, it is useful to show an arc-weighted generalization of this statement as follows.
\begin{lemma}\label{lemma:arcweights}
Let $k \in \mathbb{N}$, $t=\frac{k^2+3k}{2}$. Let $D$ be a digraph containing $\bivec{K}_t$ as a butterfly-minor, and let $w:A(D) \rightarrow \mathbb{R}_+$ be an arc-weighting of $D$. Then $D$ contains $k$ pairwise vertex-disjoint directed cycles $C_1,\ldots,C_k$ such that the total weights $w(C_i)=\sum_{e \in A(C_i)}{w(e)}$, $i=1,\ldots,k$ of the cycles are pairwise distinct.
\end{lemma}
\begin{proof}
By induction on $v(D)+a(D)$. We clearly have $v(D) \ge t$ and $a(D) \ge t(t-1)$ for every $D$ as in the lemma, so in the base case we have $v(D)+a(D)=t^2$ and $D=\bivec{K}_t$. Since $t=2+3+\cdots+k+(k+1)$, we can partition $V(D)$ into subsets $V_1,\ldots,V_k$, such that $|V_i|=i+1$ for $1 \le i \le k$. We claim that $D[V_i]$ contains at least $i$ directed cycles with pairwise distinct total weights, for every $i \in \{1,\ldots,k\}$. Indeed, for $1 \le i \le k$ let us pick some vertex $v_i \in V_i$ and order the vertices in $V_i\setminus \{v_i\}$ as $v_{i,1},\ldots,v_{i,i}$ in such a way that $$w((v_{i,1},v_i)) \le w((v_{i,2},v_i)) \le w((v_{i,3},v_i)) \le \ldots \le w((v_{i,i},v_i)).$$ For $1 \le j \le i$, let us denote by $C_{i,j}$ the directed cycle in $D[V_i]$ induced by the arcs $(v_i,v_{i,1}),$ $(v_{i,1},v_{i,2}),$ $\ldots,$ $(v_{i,j-1},v_{i,j}),$ $(v_{i,j},v_i)$. We then have $$w(C_{i,j})-w(C_{i,j-1})=\underbrace{w((v_{i,j-1},v_{i,j}))}_{>0}+\underbrace{(w((v_{i,j},v_i))-w((v_{i,j-1},v_i)))}_{\ge 0}>0$$ for $2 \le j \le i$, showing that $C_{i,1},\ldots,C_{i,i}$ are $i$ directed cycles in $D[V_i]$ of pairwise distinct total weights. For $i=1,2,\ldots,k$ we can now successively pick a cycle $C_i \in \{C_{i,j}|1 \le j \le i\}$ whose weight is distinct from the weights of the already chosen cycles $C_1,\ldots,C_{i-1}$. This proves the claim in the base case of the induction.

Suppose now that $v(D)+a(D)>t^2$ and that the claim holds for all digraphs $D'$ such that $v(D')+a(D')<v(D)+a(D)$. Then either there is a proper subdigraph $D_1 \subsetneq D$ such that $D_1$ still contains $\bivec{K_t}$ as a butterfly-minor, or $D$ contains a contractible arc $e=(u,v) \in A(D)$ such that $D_2:=D/e$ contains $\bivec{K_t}$ as a butterfly-minor. In the first case, we can apply the induction hypothesis to $D_1$ (with arc-weights inherited from $D$) to find that there exists $k$ vertex-disjoint directed cycles in $D_1$ (and hence also in $D$) of pairwise distinct weights. In the second case, we know that either $d_D^+(u)=1$ or $d_D^-(v)=1$. Let us assume w.l.o.g. that $e$ is the only arc leaving $u$, as the proof in the other case works symmetrically. Then $V(D_2)=V(D)\setminus \{u\}$ and $A(D_2)=A(D-u) \cup \{(x,v)|(x,u) \in A(D), x \neq v\}$. Let us define an arc-weighting $w_2:A(D_2) \rightarrow \mathbb{R}_+$ by $w_2(e):=w(e)$ for all $e \in A(D-u)$ and $w_2((x,v)):=w((x,u))+w((u,v))$ for all $x \in N_D^-(u)\setminus \{v\}$. Since $v(D_2)+a(D_2)<v(D)+a(D)$, the induction hypothesis yields that $D_2$ contains $k$ vertex-disjoint directed cycles $C_1,\ldots,C_k$ with pairwise distinct total weights. If none of these cycles uses an arc of the form $(x,v)$ with $x \in N_D^-(u)\setminus \{v\}$, then the cycles also exist in $D$ with the same weights, and hence the inductive claim holds. Otherwise, exactly one of the cycles, say $C_1$, uses an arc of the from $(x,v)$ with $x \in N_D^-(u)\setminus \{v\}$. But then replacing the arc $(x,v)$ on $C_1$ with the path $(x,u),(u,v)$ in $D$, we find a directed cycle $C_1'$ in $D$ such that $V(C_1')=V(C_1) \cup \{u\}$, and such that the weight of $C_1$ in $(D_2,w_2)$ equals the weight of $C_1'$ in $(D,w)$. Hence, $C_1',C_2,\ldots,C_k$ are $k$ vertex-disjoint directed cycles in $D$ with pairwise distinct weights also in $D$, again proving the inductive claim. This concludes the proof by induction.
\end{proof}
By putting $w(e):=1$ for all $e \in A(D)$ in Lemma~\ref{lemma:arcweights}, we obtain the following.
\begin{corollary}\label{cor:kt}
Let $D$ be a digraph containing $\bivec{K}_t$ as a butterfly-minor, where $t \ge \frac{k^2+3k}{2}$. Then $D$ contains $k$ vertex-disjoint directed cycles of pairwise distinct lengths.
\end{corollary}
\subsection{Digraphs Containing a Flat Wall}\label{sec:flatwall}
In this subsection we give conditions which guarantee many vertex-disjoint directed cycles of distinct lengths in digraphs of large minimum degree containing a weakly flat wall of large order. Let us start by introducing a useful notation: Let $D$ be a digraph, and let $W \subseteq D$ be a cylindrical wall. For a vertex $w \in \textbf{int}(W)$, in the rest of this section we denote by $R_W^+[w]$ and $R_W^-[w]$ the sets of vertices in $V(D)\setminus \textbf{int}(W)$ which are reachable from $w$ in $D-(\textbf{int}(W) \setminus \{w\})$, respectively which can reach $w$ in $D-(\textbf{int}(W) \setminus \{w\})$.

\begin{lemma}\label{lemma:in or out}
Let $D$ be a digraph containing a cylindrical wall $W \subseteq D$. Let $a,b \in \mathbb{N}$ such that $\delta^+(D) \ge a+b$, and let $w \in \textbf{int}(W)$. Then at least one of the following holds.
\begin{itemize}
\item There exists $x \in V(D)$, a directed path $P$ in $D$ from $w$ to $x$ such that $V(P) \cap \textbf{int}(W)=\{w\}$ and distinct vertices $v_1,\ldots,v_a \in \textbf{int}(W)\setminus \{w\}$ such that $(x,v_i) \in A(D), i=1,\ldots,a$, or
\item $D[R_W^+[w]]$ has minimum out-degree at least $b$.
\end{itemize}
\end{lemma}
\begin{proof}
If $R_W^+[w]=\emptyset$, then putting $x=w$ verifies the first item. Now suppose $R_W^+[w] \neq \emptyset$.
By definition of $R_W^+[w]$, for every vertex $x \in R_W^+[w]$, we have $N_D^+(x) \setminus \textbf{int}(W) \subseteq R_W^+[w]$. Now suppose that $D[R_W^+[w]]$ has minimum out-degree less than $b$. Then there exists $x \in R_W^+[w]$ such that $$b>d^+_{D[R_W^+[w]]}(x)=|N_D^+(x) \setminus \textbf{int}(W)|$$ $$=d_D^+(x)-|N_D^+(x) \cap \textbf{int}(W)| \ge a+b-|N_D^+(x) \cap \textbf{int}(W)|.$$ Consequently, $|N_D^+(x) \cap (\textbf{int}(W) \setminus \{w\})| \ge |N_D^+(x) \cap \textbf{int}(W)|-1 \ge a$, and we find that $x$ has $a$ distinct out-neighbors $v_1,\ldots,v_a \in \textbf{int}(W) \setminus \{w\}$. Since $x \in R_W^+[w]$, there exists a dipath $P$ from $w$ to $x$ such that $V(P) \cap \textbf{int}(W)=\{w\}$, and the claim follows.
\end{proof}
To state our next lemma we need the following definition.
\begin{definition}
Let $W$ be a cylindrical wall. Let $G_W$ denote the auxiliary undirected graph on the vertex-set $\textbf{int}(W)$ in which two vertices $x \neq y \in \textbf{int}(W)$ are adjacent if there exists a brick $B$ of $W$ such that $x,y \in V(B)$. Then we define the \emph{brick-distance} $\text{dist}_W(w_1,w_2)$ of two vertices $w_1, w_2 \in \textbf{int}(W)$ as their distance in the graph $G_W$.
\end{definition}
\begin{lemma}\label{lemma:intersections}
Let $D$ be a digraph, and let $W \subseteq D$ be a cylindrical wall which is weakly flat in $D$. Let $w_1, w_2 \in \textbf{int}(W)$. Then
\begin{enumerate}[label=(\roman*)]
\item $R_W^+[w_1] \cap R_W^-[w_2]=\emptyset$, if $\text{dist}_W(w_1,w_2) \ge 2$, and
\item $R_W^+[w_1] \cap R_W^+[w_2]=\emptyset$, if $D$ is strongly connected and $\text{dist}_W(w_1,w_2) \ge 3$.
\end{enumerate}
\end{lemma}
\begin{proof}\noindent
\begin{enumerate}[label=(\roman*)]
\item Suppose towards a contradiction that $\text{dist}_W(w_1,w_2) \ge 2$ and $R_W^+[w_1] \cap R_W^-[w_2] \neq \emptyset$. Pick some $x \in R_W^+[w_1] \cap R_W^-[w_2]$. Then by definition, there exists a $w_1$-$x$-dipath $P_1$ and an $x$-$w_2$-dipath $P_2$ in $D$ such that $V(P_1) \cap \textbf{int}(W)=\{w_1\}$, $V(P_2)\cap\textbf{int}(W)=\{w_2\}$. Let $z$ be the first vertex of $V(P_2)$ we meet when traversing $P_1$ starting at $w_1$. Then $P:=P_1[w_1,z]\circ P_2[z,w_2]$ is a $w_1$-$w_2$-dipath in $D$ satisfying $V(P) \cap \textbf{int}(W)=\{w_1,w_2\}$. Since $W$ is weakly flat in $D$, this means that there is a brick $B$ of $W$ such that $w_1, w_2 \in V(B)$, and hence $\text{dist}_W(w_1,w_2) \le 1$, a contradiction to our initial assumptions. This proves the claim.
\item Suppose towards a contradiction that $R_W^+[w_1] \cap R_W^+[w_2] \neq \emptyset$, that $D$ is strongly connected and that $\text{dist}_W(w_1,w_2) \ge 3$. Pick a vertex $x \in R_W^+[w_1] \cap R_W^+[w_2]$. By definition, there exists a $w_1$-$x$-dipath $P_1$ and a $w_2$-$x$-dipath $P_2$ in $D$ such that $V(P_1) \cap \textbf{int}(W)=\{w_1\}$, $V(P_2) \cap \textbf{int}(W)=\{w_2\}$. Since $D$ is strongly connected, there exists a directed path in $D$ starting in $x$ and ending in a vertex of $\text{int}(W)$. Let $P_3$ be a shortest directed path with this property. Then clearly $V(P_3) \cap \textbf{int}(W)=\{y\}$, where $y$ denotes the last vertex of $P_3$. Since $3 \le \text{dist}_W(w_1,w_2) \le \text{dist}_W(w_1,y)+\text{dist}_W(y,w_2)$, we have $\text{dist}_W(w_1,y) \ge 2$ or $\text{dist}_W(w_2,y) \ge 2$. W.l.o.g. suppose that the former is true. Let $z$ be the first vertex of $V(P_3)$ we meet when traversing $P_1$ starting at $w_1$. Then $P:=P_1[w_1,z] \circ P_3[z,y]$ is a $w_1$-$y$-dipath in $D$ satisfying $V(P) \cap \textbf{int}(W)=\{w_1,y\}$. Since $W$ is weakly flat in $D$, we conclude that there exists a brick $B$ of $W$ such that $w_1, y \in V(B)$, contradicting the fact that $\text{dist}_W(w_1,y) \ge 2$. Hence our initial assumption was wrong, and the proof concludes.
\end{enumerate}
\end{proof}

The next lemma is the main technical result of this subsection, yielding conditions which guarantee a $k$-train intersecting a weakly flat wall only in a ``thin strip''.
\begin{lemma}\label{lemma:technical}
Let $k \in \mathbb{N}$, let $D$ be a digraph with $\delta^+(D) \ge 7k-5$, and let $W \subseteq D$ be a cylindrical wall which is weakly flat in $D$. Let $w$ be a branch-vertex of $D$, whose coordinates in $W$ are $(c_1,c_2)$ for $c_1, c_2 \in \mathbb{N}$ such that $c_1=2\ell-1$ is odd and $c_2$ is even. 

Let $S \subseteq V(W)$ be defined as the set consisting of all branch vertices of $W$ whose first coordinate is in $\{c_1-2,c_1-1,c_1,c_1+1,c_1+2,c_1+3\}$, together with all vertices contained in the interior of a subdivision-path of $W$ spanned between two such branch vertices.

If $S \cap \textbf{per}(W)=\emptyset$, then $D[S \cup R_W^+[w]]$ contains a $k$-train.
\end{lemma}
\begin{proof}
\begin{figure}[h]
\centering
\includegraphics[scale=1]{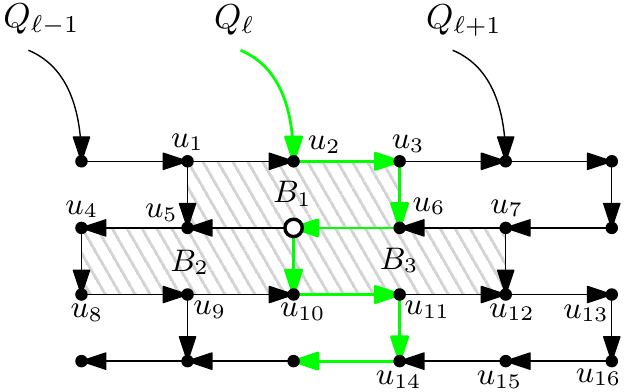}
\caption{The local situation around the vertex $w$ in the proof of Lemma~\ref{lemma:technical}, showing the bricks incident to $w$ and relevant branch vertices within $S$. The non-filled central vertex represents $w$, while the green fat path represents part of the cycle $Q_{\ell}$.}\label{fig:zoomin}
\end{figure}
By definition of $S$ and since $S \cap \textbf{per}(W)=\emptyset$, we have $w \in S \subseteq \textbf{int}(W)$. Since the definition of $S$ depends only on the first coordinate of $w$, and by the cylindrical symmetry of $W$, we may assume w.l.o.g. $c_2=2$. 
Note that $w$ is contained in the directed cycle $Q_{\ell}$ of $W$, and that $V(Q_{\ell-1}) \cup V(Q_\ell) \cup V(Q_{\ell+1}) \subseteq S$.
 
By Lemma~\ref{lemma:in or out} applied with $a=6k-5$ and $b=k$, either (1) there exists $x \in V(D)\setminus \textbf{int}(W)$, a $w$-$x$-dipath $P$ in $D$ such that $V(P) \cap \textbf{int}(W)=\{w\}$, and distinct out-neighbors $v_1,\ldots,v_{6k-5} \in \textbf{int}(W)\setminus \{w\}$ of $x$, or (2) $\delta^+(D[R_W^+[w]]) \ge k$.

If (2) occurs, then we can apply Observation~\ref{obs:existsktrain} to conclude that $D[R_W^+[w]]$ and therefore also $D[S \cup R_W^+[w]]$ contains a $k$-train as a subdigraph, which proves the assertion in this case. Hence, assume for the rest of the proof that (1) occurs. 

Note that $V(P) \cap \textbf{int}(W)=\{w\}$ implies that $x \in V(P) \subseteq \{w\} \cup R_W^+[w] \subseteq S \cup R_W^+[w]$. Further note that for every $i \in \{1,\ldots,6k-5\}$, the dipath $P \circ (x,v_i)$ intersects $\textbf{int}(W)$ only in its first and last vertex, and hence the weak flatness of $W$ in $D$ implies that $v_i$ is contained in one of the three bricks of $W$ meeting at $w$. 

For further reasoning, we fix the following notation (compare Figure~\ref{fig:zoomin}): We denote by $u_1,\ldots,u_{16}$ $16$ distinct branch vertices of $W$ contained in $S$, given by the following coordinates: $$(c_1-1,1),(c_1,1),(c_1+1,1),(c_1-2,2),(c_1-1,2),(c_1+1,2),(c_1+2,2),$$ $$(c_1-2,3),(c_1-1,3),(c_1,3),(c_1+1,3),(c_1+2,3),(c_1+3,3),(c_1+1,4), (c_1+2,4),(c_1+3,4).$$
Let us label the three bricks incident to $w$ as $B_1, B_2, B_3$, where $B_1$ contains the branch vertices $u_1, u_2, u_3, u_5, w, u_6$, $B_2$ contains $u_4, u_5, w, u_8, u_9, u_{10}$, and $B_3$ contains $w, u_6, u_7, u_{10}, u_{11}, u_{12}$.

In the following, for two distinct branch-vertices $s, t$ of $W$ which are connected by a directed subdivision-path, let us denote this path by $W[s,t]=W[t,s]$. 

Let $P_1, P_2, P_3, P_4, P_5, P_6$ be $6$ directed paths in $S_i$ defined by
$$P_1=W[u_1,u_2] \circ W[u_2,u_3] \circ W[u_3,u_6] \circ W[u_6, w], P_2=W[u_1,u_5],$$
$$P_3=W[w,u_5] \circ W[u_5, u_4] \circ W[u_4,u_8] \circ W[u_8, u_9] \circ W[u_9,u_{10}],$$
$$P_4=W[w,u_{10}]\circ W[u_{10},u_{11}] \circ W[u_{11},u_{12}],$$
$$P_5=W[u_7,u_6], P_6=W[u_7,u_{12}].$$

Then we have $V(B_1) \cup V(B_2) \cup V(B_3)=V(P_1) \cup V(P_2) \cup V(P_3) \cup V(P_4) \cup V(P_5) \cup V(P_6)$. By the above, we have $\{v_1,\ldots,v_{6k-5}\} \subseteq (V(B_1) \cup V(B_2) \cup V(B_3))\setminus \{w\}$, and hence there is $j \in \{1,\ldots,6\}$ such that $P_j$ contains at least $k$ of the vertices $v_1,\ldots,v_{6k}$. Let $y_1,\ldots,y_k \in V(P_j) \setminus \{w\}$ be the ordering of these $k$ vertices on the path $P_j$ when traversing it starting from its first vertex. The rest of the proof shows how to find a $k$-train in $D[S \cup R_W^+[w]]$ for every value of $j$.
\begin{itemize}
\item \textbf{Case $j=1$}. Let us consider the directed path $Q:=P_1[y_1,w] \circ P$, which starts at $y_1$ and ends at $x$. Then $Q \subseteq D[S \cup R_W^+[w]]$, and the last vertex $x$ of $Q$ has the $k$ distinct out-neighbors $y_1,\ldots,y_k$ on $Q$, showing that $D[S \cup R_W^+[w]]$ contains a $k$-train.

\item \textbf{Case $j=2$}. Consider the directed path $Q:=P_2[y_1,u_5] \circ P_3[u_5,u_{10}] \circ Q_{\ell}[u_{10},w] \circ P$, which starts at $y_1$ and ends at $x$ (recall that $Q_{\ell}[u_{10},w]$ denotes the directed subpath of the directed cycle $Q_{\ell}$ starting at $u_{10}$ and ending at $w$). Then $Q \subseteq D[S \cup R_W^+[w]]$, and the last vertex $x$ of $Q$ has the $k$ distinct out-neighbors $y_1,\ldots,y_k$ on $Q$. Hence, $D[S \cup R_W^+[w]]$ contains a $k$-train.

\item \textbf{Case $j=3$}. Consider the directed path $Q:=P_3[y_1,u_{10}] \circ Q_{\ell}[u_{10},w] \circ P$, which starts at $y_1$ and ends at $x$. Then $Q \subseteq D[S \cup R_W^+[w]]$, and the last vertex $x$ of $Q$ has the $k$ distinct out-neighbors $y_1,\ldots,y_k$ on $Q$, showing that $D[S \cup R_W^+[w]]$ contains a $k$-train.

\item \textbf{Case $j=4$}. Consider the directed path $$Q:=P_4[y_1,u_{12}] \circ W[u_{12},u_{13}] \circ W[u_{13},u_{16}] \circ W[u_{16}, u_{15}] \circ W[u_{15},u_{14}] \circ Q_\ell[u_{14},w] \circ P,$$ which starts at $y_1$ and ends at $x$. Then $Q \subseteq D[S \cup R_W^+[w]]$, and the last vertex $x$ of $Q$ has the $k$ distinct out-neighbors $y_1,\ldots,y_k$ on $Q$, showing that $D[S \cup R_W^+[w]]$ contains a $k$-train.

\item \textbf{Case $j=5$}. Consider the directed path $Q:=P_5[y_1,u_6] \circ W[u_6, w] \circ P$, which starts at $y_1$ and ends at $x$. Then $Q \subseteq D[S \cup R_W^+[w]]$, and the last vertex $x$ of $Q$ has the $k$ distinct out-neighbors $y_1,\ldots,y_k$ on $Q$, showing that $D[S \cup R_W^+[w]]$ contains a $k$-train.

\item \textbf{Case $j=6$}. Consider the directed path $$Q:=P_6[y_1,u_{12}] \circ W[u_{12},u_{13}] \circ W[u_{13},u_{16}] \circ W[u_{16}, u_{15}] \circ W[u_{15},u_{14}] \circ Q_\ell[u_{14},w] \circ P,$$ which starts at $y_1$ and ends at $x$. Then $Q \subseteq D[S \cup R_W^+[w]]$, and the last vertex $x$ of $Q$ has the $k$ distinct out-neighbors $y_1,\ldots,y_k$ on $Q$, showing that $D[S \cup R_W^+[w]]$ contains a $k$-train.
\end{itemize}

In every case, $D[S \cup R_W^+[w]]$ contains a $k$-train. This concludes the proof of the lemma.
\end{proof}

Before moving on, let us note the following symmetrical version of Lemma~\ref{lemma:technical}. In the following we refer to a digraph obtained from a $k$-train by reversing the orientations of all its arcs as a \emph{reverse-$k$-train}.
\begin{corollary}\label{cor:sym}
Let $k \in \mathbb{N}$, let $D$ be a digraph with $\delta^-(D) \ge 7k-5$, and let $W \subseteq D$ be a cylindrical wall which is weakly flat in $D$. Let $w$ be a branch-vertex of $D$, whose coordinates in $W$ are $(c_1,c_2)$ for $c_1, c_2 \in \mathbb{N}$ such that $c_1=2\ell-1$ and $c_2$ are odd. 

Let $S \subseteq V(W)$ be defined as the set consisting of all branch vertices of $W$ whose first coordinate is in $\{c_1-2,c_1-1,c_1,c_1+1,c_1+2,c_1+3\}$, together with all vertices contained in the interior of a subdivision-path of $W$ spanned between two such branch vertices.

If $S \cap \textbf{per}(W)=\emptyset$, then $D[S \cup R_W^-[w]]$ contains a reverse-$k$-train.
\end{corollary}
\begin{proof}
Let $m$ denote the order of the wall $W$. Let $\revvec{D}$ denote the digraph obtained from $D$ by reversing its arcs. Then $\delta^+(\revvec{D}) \ge 7k-5$ and $\revvec{D}$ contains $\revvec{W}$, the digraph obtained from $W$ by reversing the orientations, as a subdigraph. Then $\revvec{W}$ itself forms a wall with the following associated coordinates: Any branch-vertex of $W$ at coordinates $(t_1,t_2)$ receives the coordinates $(t_1,2m-t_1)$ in $\revvec{W}$. In particular $w$ has cordinates $(c_1,2m+1-c_2)$ in $\revvec{W}$, where $c_1$ is odd and $2m+1-c_2$ is even. Since $\textbf{per}(W)=\textbf{per}(\revvec{W})$, the weak flatness of $W$ in $D$ implies that also $\revvec{W}$ is weakly flat in $\revvec{D}$. Also note that the set $S$ remains  invariant when changing between $W$ and $\revvec{W}$. We may thus apply Lemma~\ref{lemma:technical} to $\revvec{D}$ and the vertex $w$ in $\revvec{W} \subseteq \revvec{D}$ to find that $\revvec{D}[S \cup R_{\revvec{W}}^+[w]]=\revvec{D}[S \cup R_{W}^-[w]]$ contains a $k$-train. It directly follows that $D[S \cup R_W^-[w]]$ contains a reverse-$k$-train, and the claim follows.
\end{proof}

The following consequences of Lemma~\ref{lemma:technical} are used in the proofs of Theorem~\ref{thm:mainconn} and Theorem~\ref{thm:mainsem}. 

\begin{lemma}\label{lemma:strongcase}
Let $D$ be a strongly connected digraph, and let $W \subseteq D$ be a cylindrical wall of order $3k+2$ which is weakly flat in $D$. If $\delta^+(D) \ge 7k-5$, then $D$ contains $k$ vertex-disjoint directed cycles of distinct lengths.
\end{lemma}
\begin{proof}
Let us pick distinct branch vertices $w_1,\ldots,w_{k} \in \textbf{int}(W)$ of $W$ given by the following coordinates (compare Figure~\ref{fig:bigwall}): For every $1 \le i \le k$, $w_i$ is at position $(6i-1,2)$ in $W$. We note that $w_i$ is contained in the cycle $Q_{3i}$ of $W$ for every $1 \le i \le k$, and that $\text{dist}_W(w_i,w_j) \ge 3$ for every $1 \le i <j \le k$. Using Lemma~\ref{lemma:intersections}, (ii) we conclude that the sets $R_W^+[w_i]$, $i=1,\ldots,k$ are pairwise disjoint. For every $1 \le i \le k$, let us define the \emph{$i$-th strip} $S_i \subseteq \textbf{int}(W)$ as the set of vertices $v \in V(W)$, such that either $v$ is a branch vertex of $W$ whose first coordinate is in $\{6i-3,6i-2,\ldots,6i+2\}$, or such that $v$ is contained in a subdivision-path of $W$ between two branch vertices in $S_i$. Note that the sets $S_1,\ldots,S_k$ are pairwise vertex-disjoint. For every $i=1,\ldots,k$ we can now apply Lemma~\ref{lemma:technical} to $D$ and the vertex $w_i$ in $W$, and we find that $D[S_i \cup R_W^+[w_i]]$ contains a $k$-train. Since the sets $S_i \cup R_W^+[w_i], i=1,\ldots,k$ are pairwise disjoint, it follows that $D$ contains $k$ vertex-disjoint $k$-trains. 
The assertion now follows by applying Observation~\ref{obs:disjointktrains} to $D$.
\end{proof}
\begin{figure}[h]
\centering
\includegraphics[scale=0.86]{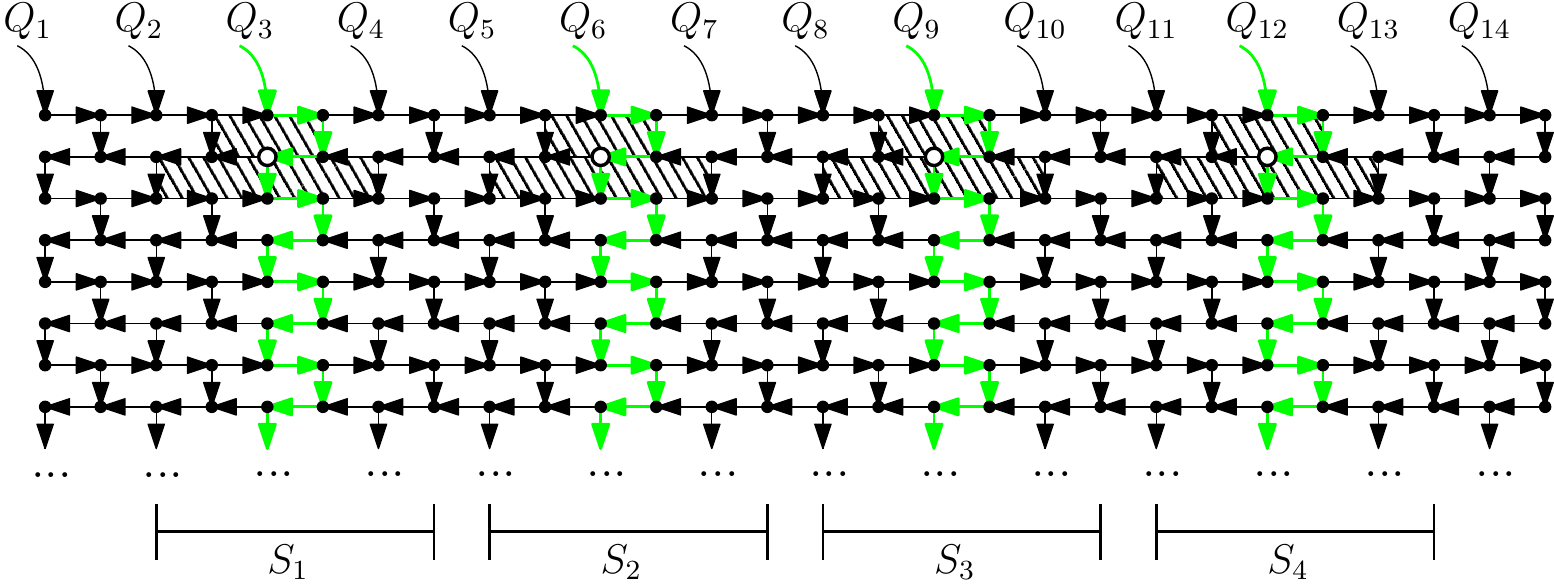}
\caption{Illustration of the placement of the vertices $w_1,\ldots,w_{4}$ in the wall $W_{14}$ in the proof of Lemma~\ref{lemma:strongcase}. The vertices $w_i$ are marked by big non-filled vertices, and the three incident bricks of every $w_i$ are shaded. The directed cycles $Q_{3i}, i=1,\ldots,4$ containing the vertices $w_i$ are marked fat and green.}\label{fig:bigwall}
\end{figure}

\begin{lemma}\label{lemma:nonstrongcase}
Let $D$ be a digraph, and let $W \subseteq D$ be a cylindrical wall of order $8$ which is weakly flat in $D$. Suppose further that $\delta^+(D),\delta^-(D) \ge 16$. Then $D$ contains $3$ vertex-disjoint directed cycles of distinct lengths.
\end{lemma}
\begin{proof}
Let us place two distinct vertices $w_1, w_2$ in the interior of $W$ as follows: $w_1$ is the branch-vertex of $W$ with coordinates $(5,2)$, and $w_2$ is the branch vertex of $W$ with coordinates $(11,3)$. Then we have $\text{dist}_W(w_1,w_2)=3 \ge 2$, and by Lemma~\ref{lemma:intersect}, (i) we have $R_W^+[w_1] \cap R_W^-[w_2]=\emptyset$. Let the strips $S_1$, $S_2 \subseteq \textbf{int}(W)$ be defined as the sets consisting of the branch-vertices of $W$ with first coordinates in $\{3,4,\ldots,8\}$ respectively $\{9,10,\ldots,14\}$ as well as the vertices on the subdivision-paths of $W$ spanned between those vertices. It now follows from Lemma~\ref{lemma:technical} applied to $w_1$ with $k=3$, respectively from Corollary~\ref{cor:sym} applied to $w_2$ with $k=3$, that $D[S_1 \cup R_W^+[w_1]]$ contains a $3$-train and that $D[S_2 \cup R_W^-[w_2]]$ contains a reverse-$3$-train. Noting that $V(Q_1)$, $S_1 \cup R_W^+[w_1]$ and $S_2 \cup R_W^-[w_2]$ are pairwise disjoint, we conclude that $D$ contains the directed cycle $Q_1$, a $3$-train, and a reverse $3$-train as subdigraphs, pairwise vertex-disjoint. Since every $3$-train and every reverse $3$-train contains $3$ directed cycles of distinct lengths, it follows that $D$ contains $3$ vertex-disjoint directed cycles of distinct lengths.
\end{proof}
\subsection{Proofs of Theorem~\ref{thm:mainconn} and Theorem~\ref{thm:mainsem}}\label{sec:thm12}
We are now ready to give the proofs of Theorem~\ref{thm:mainconn} and Theorem~\ref{thm:mainsem}.
\begin{proof}[Proof of Theorem~\ref{thm:mainconn}]
Let $k \in \mathbb{N}$ and put $s(k):=\max\{d(3k+2,t),a(t)+7k-5\}$, where $t:=\frac{k^2+3k}{2}$ and $d(\cdot,\cdot), a(\cdot)$ are the functions from Theorem~\ref{thm:flatwallstrong}. 

Let $D$ be any given strongly $s(k)$-connected digraph. By Theorem~\ref{thm:flatwallstrong}, applied to $D$ with parameters $3k+2$ and $t$, we find that one of the following must be true.
\begin{enumerate}[label=(\roman*)]
\item $\dtw(D)<d(3k+2,t)$, or
\item $D$ contains $\bivec{K}_t$ as a butterfly-minor, or
\item there exists $X \subseteq V(D)$ with $|X| \le a(t)$ and a cylindrical wall $W \subseteq D-X$ of order $3k+2$ which is weakly flat in $D-X$.
\end{enumerate}
Case (i) is impossible, since we have $\dtw(D) \ge s(k) \ge d(3k+2,t)$ by Corollary~\ref{cor:highconnhighwidth}. If Case (ii) occurs, then we can apply Corollary~\ref{cor:kt} to $D$ and we find that $D$ contains $k$ vertex-disjoint directed cycles of pairwise distinct length, as required. Finally, if Case (iii) occurs, we know from $s(k) \ge a(t)+7k-5$ that $D-X$ is strongly $(7k-5)$-connected. This in particular implies that $D-X$ is strongly connected and $\delta^+(D) \ge 7k-5$. Since there exists a cylindrical wall $W \subseteq D-X$ of order $3k+2$ which is weakly flat in $D-X$, we can apply Lemma~\ref{lemma:strongcase} to $D-X$ and find that there exist $k$ vertex-disjoint directed cycles of pairwise distinct lengths in $D-X \subseteq D$, yielding the assertion also in this case. This concludes the proof.
\end{proof}
\begin{proof}[Proof of Theorem~\ref{thm:mainsem}]
Let $K:=\max\{2d(8,9)+3,a(9)+16\}$, where $d(\cdot,\cdot), a(\cdot)$ are the functions from Theorem~\ref{thm:flatwallstrong}. 

Let $D$ be any given digraph such that $\delta^+(D), \delta^-(D) \ge K$. By applying Theorem~\ref{thm:flatwallstrong} to $D$ with parameters $8$ and $9$, we find that one of the following must hold:
\begin{enumerate}[label=(\roman*)]
\item $\dtw(D) < d(8,9)$, or
\item $D$ contains $\bivec{K}_9$ as a butterfly-minor, or
\item There exists $X \subseteq V(D)$ with $|X| \le a(9)$ such that $D-X$ contains a wall $W$ of order $8$ which is weakly flat in $D-X$.
\end{enumerate}
If Case (i) occurs, then we have $\dtw(D) \le d(8,9)-1$. Since $\delta^+(D) \ge K \ge 2d(8,9)+3$ $=(d(8,9)+1)(3-1)+1>(d(8,9)+1)(3-1)$, we can apply Proposition~\ref{prop:boundedwidth} to $D$ with parameters $d:=d(8,9)-1$ and $k=3$ and find that indeed, $D$ contains $3$ vertex-disjoint directed cycles of distinct lengths in this case, proving the assertion in this case.
If Case (ii) occurs, then we can apply Corollary~\ref{cor:kt} to $D$ with $t=9, k=3$ and find that $D$ contains $3$ vertex-disjoint directed cycles of distinct lengths.
Finally, if (iii) occurs, then we have $\delta^+(D-X) \ge \delta^+(D)-|X| \ge a(9)+16-a(9)=16$ and $\delta^-(D-X) \ge \delta^+(D)-|X| \ge a(9)+16-a(9)=16$. Since $D-X$ contains a weakly flat wall $W$ of order $8$, it follows from Lemma~\ref{lemma:nonstrongcase} that $D-X \subseteq D$ contains $3$ vertex-disjoint directed cycles of distinct lengths also in this case. This concludes the proof.
\end{proof}
\section{Equicardinal Disjoint Directed Cycles}\label{sec:equicardinal}
In this section we give the proof of Proposition~\ref{prop:notwoofsamelength}, by constructing for every $k \in \mathbb{N}$ a strongly $k$-connected digraph containing no two arc-disjoint directed cycles of equal length.
\begin{proof}[Proof of Proposition~\ref{prop:notwoofsamelength}]
Let $k \in \mathbb{N}$. Let $N:=4^{k^2}$ and let $a(1),\ldots,a(k^2), b(1),\ldots,b(k^2) \in \mathbb{N}$ be defined as $a(\ell):=N+2^{\ell-1}, b(\ell):=N+2^{k^2+\ell-1}$, $1 \le \ell \le k^2$. Note that all possible subset-sums of $\{a(1),a(2),\ldots,a(k^2),b(1),b(2),\ldots,b(k^2)\}$ are pairwise distinct.

Let $D_k$ denote a digraph on $2Nk$ vertices whose vertex-set is partitioned into $2N$ disjoint parts $V_1,\ldots,V_{2N}$, each of size exactly $k$, and having the following arcs:

For every $2 \le \ell \le N$, $D_k$ contains all possible arcs of the form $(u,v), u \in V_\ell, v \in V_{\ell-1}$.

Label the vertices in $V_1$ and $V_N$ as $V_1=\{u_1,\ldots,u_k\}, V_N=\{w_1,\ldots,w_k\}$. For every $i \in \{1,\ldots,k\}$, $u_i$ has $k$ distinct out-arcs $e_{i,j}, j=1,\ldots,k$, where $e_{i,j}$ connects $u_i$ to a vertex chosen arbitrarily from $V_{a(k(i-1)+j)}$. Similarly, $v_i$ has $k$ distinct in-arcs $f_{i,j}, j=1,\ldots,k$, where $f_{i,j}$ connects a vertex chosen arbitrarily from $V_{2N-b(k(i-1)+j)+1}$ to $v_i$.

This concludes the description of $D_k$. Please note that all out-neighbors of the vertices in $V_1$ are contained in $V_{N+1}\cup \cdots \cup V_{2N}$ while all the in-neighbors of the vertices in $V_{2N}$ are contained in $V_1 \cup \cdots \cup V_{N}$. We next show that $D_k$ is strongly $k$-vertex-connected. To this end, let $S \subseteq V(D_k)$ with $|S|<k$ be arbitrary and let us prove that $D_k-S$ is strongly connected. Since $|S|<k$, we have $V_\ell \setminus S \neq \emptyset, 1 \le \ell \le 2N$, and hence in $D_k-S$ every vertex in $V_\ell\setminus S$ can reach every vertex in $V_{\ell'}\setminus S$ provided $1 \le \ell' <\ell \le 2N$. For strong connectivity of $D_k-S$ it now is sufficient to check that for every $i, j \in \{1,\ldots,k\}$ such that $u_i \in V_1 \setminus S, w_j \in V_{2N} \setminus S$, there exists a directed path $P$ in $D_k-S$ starting at $u_i$ and ending in $w_j$. Since $u_i$ and $w_j$ have out-degree respectively in-degree $k$ in $D_k$, there exist vertices $x, y \in V(D_k)\setminus S$ such that $(u_i,x),(y,w_j) \in A(D_k)$. Note that by construction, $x \in V_\ell \setminus S, y \in V_{\ell'}\setminus S$ for some $\ell \in \{N+1,\ldots,2N\}$, $\ell' \in \{1,\ldots,N\}$. Hence, from the above it follows that $x$ can reach $y$ in $D_k-S$, and hence also $u_i$ can reach $w_j$ in $D_k-S$. This shows that $D_k-S$ is indeed strongly connected for every $S \subseteq V(D_k), |S|<k$, and hence that $D_k$ is strongly $k$-vertex-connected.

Let us now show that $D$ does not contain two arc-disjoint directed cycles of equal lengths. To this end, let us call $E:=\{e_{i,j},f_{i,j}|1 \le i,j \le k\}$ the set of \emph{forward arcs} of $D_k$ and define the \emph{lengths} of the forward arcs as $L(e_{i,j}):=a(k(i-1)+j), L(f_{i,j}):=b(k(i-1)+j), 1 \le i,j \le k$. Let now $C$ be a directed cycle in $D_k$ using the forward arcs $e_1,\ldots,e_r \in E$ in this cyclical order, and for $1 \le s \le r$ let $1 \le h(s) < t(s) \le 2N$ be such that $e_s$ starts in $V_{t(s)}$ and ends in $V_{h(s)}$. Since all non-forward arcs on $C$ connect a vertex of $V_\ell$ to $V_{\ell-1}$ for some $2 \le \ell \le 2N$, $C-\{e_1,\ldots,e_r\}$ decomposes into directed paths of lengths $h(1)-t(2), h(2)-t(1), h(3)-t(2), \ldots, h(r-1)-t(r),h(r)-t(1)$. It follows that
$$|C|=r+(h(1)-t(2))+\ldots+(h(r-1)-t(r))+(h(r)-t(1))=\sum_{s=1}^{r}{(h(s)-t(s)+1)}.$$ Noting that $h(s)-t(s)+1=L(e_s)$ for every $1 \le s \le k$, we obtain that the length of any directed cycle in $D_k$ equals the sum of the lengths of the forward arcs it is using. The multi-set of the edge lengths of forward-arcs equals $\{a(1),\ldots,a(k^2),b(1),\ldots,b(k^2)\}$. Since the subset-sums of this set are pairwise distinct, two directed cycles in $D_k$ have the same length iff they use the same forward-arcs. This concludes the proof.
\end{proof}
\section{Conclusion}\label{sec:conc}
In this paper, we have studied the existence of vertex-disjoint directed cycles with length constraints in digraphs of large connectivity. We have found that while we are guaranteed to find many disjoint cycles of different lengths in such digraphs, we cannot even be sure to find two arc-disjoint directed cycles of equal lengths. It would be interesting to complete this picture by understanding whether further length constraints, such as parities, can be enforced on the disjoint cycles. Clearly, bipartite digraphs of large connectivity show that we cannot expect the containment of odd dicycles. As mentioned in the introduction, a result by Thomassen~\cite{thom85} states that the existence of an even directed cycle cannot be forced by large minimum out-degree and in-degree. In contrast, he showed that every strongly $3$-connected digraph contains an even directed cycle~\cite{thom92}. We propose the following strengthening of Thomassen's result.
\begin{conjecture}
For every $k \in \mathbb{N}$, there exists an integer $s(k) \in \mathbb{N}$ such that every strongly $s(k)$-connected digraph contains $k$ vertex-disjoint directed cycles of even lengths.
\end{conjecture}

\end{document}